\documentclass[letterpaper,10pt,conference]{ieeeconf}

\IEEEoverridecommandlockouts
\newtheorem{theorem}{Theorem}[section]

\newtheorem{remark}[theorem]{Remark}

\usepackage{xurl}
\usepackage{csquotes}
\usepackage{dirtytalk}
\usepackage{algorithm,algorithmic}
\usepackage{amssymb}
\usepackage{graphicx}
\usepackage[cmex10]{amsmath}
\usepackage{array}
\usepackage{mdwtab}
\usepackage{fixltx2e}
\usepackage{graphics}
\usepackage{epsfig}
\usepackage{times}
\usepackage{amsmath}
\usepackage{amssymb}
\usepackage{url}
\usepackage{lscape}
\usepackage{color}
\usepackage{epsfig}
\usepackage{cite}
\usepackage{balance}
\usepackage{dblfloatfix}
\usepackage{hyperref}
\usepackage{caption}
\usepackage{subcaption}

\allowdisplaybreaks

\title{\LARGE \bf A Guaranteed-Stable Neural Network Approach\\ for Optimal Control of Nonlinear Systems}

\author{Anran Li, John P. Swensen, and Mehdi Hosseinzadeh,~\IEEEmembership{Senior Member,~IEEE}
\thanks{The work of Mehdi Hosseinzadeh was supported by the WSU Voiland College of Engineering and Architecture through a start-up package.}
\thanks{The authors are with the School of Mechanical and Materials Engineering, Washington State University, Pullman, WA 99164, USA (email: anran.li@wsu.edu, john.swensen@wsu.edu, mehdi.hosseinzadeh@wsu.edu).}
}

\begin{document}

\maketitle
\thispagestyle{empty}
\pagestyle{empty}

\begin{abstract}
A promising approach to optimal control of nonlinear systems involves iteratively linearizing the system and solving an optimization problem at each time instant to determine the optimal control input. Since this approach relies on online optimization, it can be computationally expensive, and thus unrealistic for systems with limited computing resources. One potential solution to this issue is to incorporate a Neural Network (NN) into the control loop to emulate the behavior of the optimal control scheme. Ensuring stability and reference tracking in the resulting NN-based closed-loop system requires modifications to the primary optimization problem. These modifications often introduce non-convexity and nonlinearity with respect to the decision variables, which may surpass the capabilities of existing solvers and complicate the generation of the training dataset. To address this issue, this paper develops a Neural Optimization Machine (NOM) to solve the resulting optimization problems. The central concept of a NOM is to transform the optimization challenges into the problem of training a NN. Rigorous proofs demonstrate that when a NN trained on data generated by the NOM is used in the control loop, all signals remain bounded and the system states asymptotically converge to a neighborhood around the desired equilibrium point, with a tunable proximity threshold. Simulation and experimental studies are provided to illustrate the effectiveness of the proposed methodology.
\end{abstract}



\section{Introduction}\label{sec:Intro}

Nonlinear systems are a common and significant aspect of real-world control problems. Achieving control objectives in nonlinear systems requires sophisticated control strategies that can accurately capture and manage the nonlinearities inherent in these systems. One possible approach to deal with nonlinear systems is to use nonlinear control techniques (see, e.g.,  \cite{KhalilBook,SastryBook,Rawlings1994,VidyasagarBook}); however, guaranteeing stability at all operating points with nonlinear control techniques can be challenging \cite{Hunt2011NeuralNE,Sinha2021AdaptiveRM}. Another approach is to, first, use an appropriate linearization technique (see, e.g, \cite{KhalilBook,IsidoriBook}) to linearize the system, and then utilize well-established theories on linear control system (see, e.g., \cite{OgataBook,ChenBook}). A different approach is to linearize the nonlinear system at each time instant, compute the control input for the linearized system, and repeat this procedure at the next time instant; examples are iterative Linear Quadratic Regulator (LQR) \cite{Prasad2014,Boby2014,Chen2017} and nonlinear Model Predictive Control (MPC) \cite{Schwedersky2022,Igarashi2020,Berberich2022}. Despite being promising, this approach requires computational capabilities that may not be realistic for systems with limited computational resources.

One possible approach to address this challenge is to use Neural Networks (NN) in the control loop and in combination with Lyapunov theory to ensure stability of the resulting closed-loop system. In this context, \cite{nghi2021lqr} proposes combining an LQR with an online NN model to stabilize nonlinear systems, although it does not provide stability guarantees. A NN-based controller which provides a Lyapunov function for stability is presented in \cite{dai2021lyapunov}, without addressing control objectives and performance metrics. A stabilizing NN-based control law for personal aerial vehicles has been developed in \cite{jang2022dnlc,jang2023uamdyncon}, without providing stability guarantees with the developed NN in the loop. Optimization-based projection layers have been integrated into a NN-based policy in \cite{donti2020enforcing}; this method improves robustness and performance of the system, however it is limited to linear systems. A NN-based adaptive control has been developed in \cite{autenrieb2019development,li2023tool}, without providing stability and convergence proofs. A NN-based Lyapunov-induced control law is presented in \cite{ravanbakhsh2019learner}; this method stops as soon as one feasible control Lyapunov function is detected, and thus it may lead to a poor performance.

To the best of our knowledge, the above-mentioned efforts on NN-based control lacks theoretical guarantees for stability and performance. Recently, the authors have developed a provably-stable NN-based control method in \cite{onestep2024}, where the NN is trained to imitate the behavior of a one-step-ahead predictive control policy, and, despite prior work, provides strong guarantees for stability and tracking performance. Experimental results and theoretical analysis reported in \cite{onestep2024} indicate a potential path for controlling a wide range of nonlinear control systems. However, its implementation can be challenging, as the training dataset should be generated by solving a nonlinear, non-convex optimization problem for every operating point, while existing solvers may not be able to solve the corresponding optimization problem. To address this issue, we develop a Neural Optimization Machine (NOM) \cite{Chen2022,Chen2024} to solve the corresponding optimization problem, which uses the NN's built-in backpropagation algorithm to convert the problem of solving an optimization problem into the training problem. This paper analytically investigates stability and tracking properties of the closed-loop system when the NN trained on data generated by the NOM is used in the control loop.

The key contributions of this paper are: i) developing a NOM to solve the non-convex, nonlinear optimization problem detailed in \cite{onestep2024}; ii) developing a NN-based control scheme based on the NOM-generated dataset, and analytically proving its stability and convergence properties; and iii) evaluating the effectiveness of the proposed methodology through simulation and experimental studies.



\noindent\textbf{Notation:} We denote the set of real numbers by $\mathbb{R}$, the set of positive real numbers by $\mathbb{R}_{>0}$, and the set of non-negative real numbers by $\mathbb{R}_{\geq0}$. For the matrix $A$, $A\succ0$ indicates that $A$ is positive definite, and $A\succeq0$ indicates that $A$ is positive semi-definite. We denote the transpose of matrix $A$ by $A^\top$. We use $\lambda_{\max}(A)$ and $\lambda_{\min}(A)$ to indicate the largest and smallest eigenvalue of $A$, respectively. Given $x\in\mathbb{R}^n$ and $Q\in\mathbb{R}^{n\times n}$, $\left\Vert x\right\Vert_Q=\sqrt{\left\vert x^\top Qx\right\vert}$. For a function $Y(x)$, $Y(x)|_{x=x^\dag}$ indicates that the function $Y(x)$ is evaluated at $x=x^\dag$. We use $\odot$ to indicate Hadamard product. We use $\star$ to represent matrix elements that can be interpreted from the matrix structure. The mark $\ast$ is used to indicate optimality. We use $\text{diag}\{a_1,a_2,\cdots,a_n\}$ to denote a $n\times n$ matrix with diagonal entries $a_1,a_2,\cdots,a_n\in\mathbb{R}$.

\section{Problem Statement}\label{sec:PS}
\noindent\textbf{Problem Setting:} Consider the following discrete-time affine nonlinear system:
\begin{subequations}\label{eq:system}
\begin{align}
x(t+1) =& f\left(x(t)\right)+g\left(x(t)\right)u(t),\\
y(t)=&h\left(x(t),u(t)\right),
\end{align}
\end{subequations}
where $x(t)=[x_1(t)~\cdots~x_n(t)]^\top\in \mathbb{R}^n$ is the state vector at time instant $t$, $u(t)=[u_1(t)~\cdots~u_q(t)]^\top\in \mathbb{R}^q$ is the control input at time instant $t$, $y(t)=[y_1(t)~\cdots~y_m(t)]^\top\in \mathbb{R}^m$ is the output vector at time instant $t$, and $f:\mathbb{R}^n \rightarrow \mathbb{R}^n$, $g:\mathbb{R}^{n}\rightarrow \mathbb{R}^{n\times q}$, and $h:\mathbb{R}^n\times\mathbb{R}^q\rightarrow\mathbb{R}^m$ are known nonlinear functions.

Let $\mathcal{X}\subset\mathbb{R}^n$ be the operating region of the system\footnote{This paper does not aim at enforcing the operating region as a constraint. Future work will discuss how to extend the proposed methodology to enforce state and input constraints.}. We assume that the linearized system inside the operating region $\mathcal{X}$ is stabilizable; that is, the pair $(A,B)$ is stabilizable, where $A=\frac{\partial f(x)}{\partial x}$ and $B=g(x)$ for any $x\in\mathcal{X}$. We also assume that $g(x)$ is $\mu_g$ Lipschitz continuous throughout $\mathcal{X}$. 


For any given desired reference $r \in \mathbb{R}^m$, we denote the corresponding steady state and input by $\bar{x}_r$ and $\bar{u}_r$, respectively; that is,
\begin{align}\label{eq:SSconfiguration}
\bar{x}_r=f\left(\bar{x}_r\right)+g\left(\bar{x}_r\right)\bar{u}_r,~~r=h\left(\bar{x}_r,\bar{u}_r\right).
\end{align}

The $r$ is steady-state admissible if $\bar{x}_r\in\mathcal{X}$. The set of all steady-state admissible references is denoted by $\mathcal{R}\subseteq\mathbb{R}^m$.

\noindent\textbf{NN-Based One-Step-Ahead Predictive Control:} Given a desired reference $r \in \mathcal{R}$, a NN-based control scheme has been developed in \cite{onestep2024} to steer the states of system \eqref{eq:system} to the desired steady state $\bar{x}_r$ and the control input to the desired steady input $\bar{u}_r$. In the proposed control scheme, a NN is trained to imitate the behavior of the following one-step-ahead predictive control law at any time instant $t$:
\begin{subequations}\label{eq:OptimizationProblemMain}
\begin{align}\label{eq:CostFunction}
u^\ast(t),P^\ast(t)=\arg\,\min_{u,P}\,& J\big(x(t),r,u,P\big),
\end{align}
subject to the following constraints:
\begin{align}
& x^+=A_tx(t)+B_tu,\label{eq:Constraint1}\\
& P\succ0,\label{eq:Constraint2}\\
& V(x^+,r,P)-V\left(x(t),r,P\right)\leq-\theta\left\Vert x(t)-\bar{x}_r\right\Vert,\label{eq:Constraint3}
\end{align}
\end{subequations}
where $J\big(x(t),r,u,P\big):=\left\Vert x^+-\bar{x}_r\right\Vert^2_{Q_x}+\left\Vert u-\bar{u}_r\right\Vert^2_{Q_u}+V\left(\left(x(t),r,P\right)\right)^2$, $Q_x\succeq0$ ($Q_x\in\mathbb{R}^{n\times n}$) and $Q_u\succ0$ ({$Q_u\in\mathbb{R}^{q\times q}$}) are weighting matrices, $\theta\in\mathbb{R}_{>0}$ is a design parameter, $A_t$ and $B_t$ denote the linearized system matrices around the current state $x(t)$ (i.e., $A_t=\frac{\partial f(x)}{\partial x}|_{x=x(t)}$ and $B_t=g\big(x(t)\big)$), and $V(x,r,P)=\left\Vert x-\bar{x}_r\right\Vert_{P}$ is the Lyapunov function with $P=P^\top\succ0$ being the Lyapunov matrix in the following form:
\begin{align}
P = \begin{bmatrix}
p_{11} & p_{12} & p_{13} & \cdots & p_{1n} \\
\star & p_{22} & p_{23} & \cdots & p_{2n}\\
\star & \star & p_{33} & \cdots & p_{3n}\\
\vdots & \vdots & \vdots & \ddots  & \vdots \\
\star  & \star & \star & \cdots & p_{nn}
\end{bmatrix}.
\end{align}

The first term in the objective function $J\big(x(t),r,u,P\big)$ penalizes the state error, the second term penalizes the input error, and the last term penalizes the magnitude of the Lyapunov function. The constraint \eqref{eq:Constraint1} enforces the linearized dynamics, constraint \eqref{eq:Constraint2} ensures that the Lyapunov matrix $P$ is positive definite, and constraint \eqref{eq:Constraint3} is a contractive constraint \cite{de2000contractive} enforcing the stability of the closed-loop system, with $\theta\in\mathbb{R}_{>0}$ being the contraction constant.

The one-step-ahead predictive control scheme given in \eqref{eq:OptimizationProblemMain} ensures closed-loop stability and boundedness of the tracking error; see the following theorem.

\begin{theorem}[\cite{onestep2024}]
For any given $r\in\mathcal{R}$, the control law obtained from \eqref{eq:OptimizationProblemMain} ensures that 
\begin{align}\label{eq:TrackingPerformance1}
\left\Vert x(t)-\bar{x}_r\right\Vert\leq\frac{3\sqrt{\bar{\lambda}_P}\delta}{\theta},
\end{align}
where $\bar{\lambda}_P:=\sup_{t\geq0}\lambda_{\text{max}}\big(P^\ast(t)\big)$ ($\bar{\lambda}_P\in\mathbb{R}_{>0}$), and  $\delta:=\sup_{t\geq0}\left\Vert f\big(x(t)\big)-A_t x(t)\right\Vert$. 
\end{theorem}

It has been shown in \cite{onestep2024} that when a NN is used in the loop that imitates the behavior of the one-step-ahead predictive control scheme given in \eqref{eq:OptimizationProblemMain}, the resulting closed-loop system is stable and the tracking error remains bounded at all times, despite possible NN's approximation errors; the following theorem summarizes the theoretical properties of the resulting closed-loop system.

\begin{theorem}[\cite{onestep2024}]
Consider system \eqref{eq:system}, and suppose that a NN trained to imitate the behavior of the control scheme given in \eqref{eq:OptimizationProblemMain} is utilized in the control loop. Then, the tracking error $\left\Vert x(t)-\bar{x}_r\right\Vert$ remains bounded, and there exists $\vartheta\in\mathbb{R}_{>0}$ such that $\left\Vert x(t)-\bar{x}_r\right\Vert\leq\vartheta$ as $t\rightarrow\infty$.


\end{theorem}


\noindent\textbf{Challenges and Goal:} At the present stage, one of the main obstacles to a widespread diffusion of the methodology proposed in \cite{onestep2024} is that solving the optimization problem \eqref{eq:OptimizationProblemMain} and collecting training data is very challenging. More precisely, the optimization problem \eqref{eq:OptimizationProblemMain} is nonlinear and non-convex with respect to decision variables $u$ and $P$, and existing solvers may not able to solve all instances of it. The problem becomes more severe when $\theta$ is set to a large value. Thus, the trained NN gives a large approximation error, and thus the resulting NN-based control scheme yields poor performance. This paper aims at addressing this practical issue through developing a NOM to solve the optimization problem \eqref{eq:OptimizationProblemMain} and generate a reliable training dataset. Theoretical guarantee for stability and convergence of the resulting NN-based closed-loop system will be also provided.



\section{NOM For Generating the Training Dataset}\label{sec:NOM}

As mentioned above, we will develop a NOM \cite{Chen2022,Chen2024} to solve the optimization problem \eqref{eq:OptimizationProblemMain} and generate the training dataset. NOMs convert the problem of solving an optimization problem into the NN training problem. The general structure of the proposed NOM is shown in Fig.~\ref{fig:NOMIP}, where different colors of connections mean whether the parameters are fixed (orange) or trainable (blue).

\noindent\textbf{NOM Structure:} The NOM structure shown in Fig. \ref{fig:NOMIP} is built upon the objective function $J(x,r,u,p)$ given in \eqref{eq:CostFunction}. Besides the objective function, the proposed NOM consists of the NOM starting point layer, the NOM constraint layer, and the NOM output layer. The neurons in the NOM starting point layer (neurons shown in yellow in Fig. \ref{fig:NOMIP}) and the input layer (neurons shown in gray in Fig. \ref{fig:NOMIP}) are connected in a one-to-one correspondence manner; no activation functions are applied, and the input parameters are functions of the NOM starting point with the parameters being weights and biases between the starting point layer and the input layer. We use $\omega_{u,i},~i\in\{1,\cdots,q\}$ and $\omega_{P,i},~i\in\{1,\cdots,\frac{n(n+1)}{2}\}$ to denote weights, and $\beta_{u,i},~i\in\{1,\cdots,q\}$ and $\beta_{P,i}~,i\in\{1,\cdots,\frac{n(n+1)}{2}\}$ to denote biases. Thus, given the starting point $\hat{u}$ and $\hat{P}$, the input parameters are:
\begin{subequations}
\begin{align}
{u}=&{\Omega_u}\odot\hat{u}+{\mathcal{B}_u},\\
{P}=&{\Omega_P}\odot\hat{P}+{\mathcal{B}_P},
\end{align}   
\end{subequations}
where 
\begin{align*}
&\Omega_u=\begin{bmatrix}\omega_{u,1}\\\omega_{u,2}\\\vdots\\\omega_{u,q}\end{bmatrix},~\mathcal{B}_u=\begin{bmatrix}\beta_{u,1}\\\beta_{u,2}\\\vdots\\\beta_{u,q}\end{bmatrix},\\
&\Omega_P=\begin{bmatrix}
\omega_{P,1} & \omega_{P,2} & \omega_{P,3} & \cdots & \omega_{P,n} \\
\star & \omega_{P,(n+1)} & \omega_{P,(n+2)} & \cdots & \omega_{P,(2n-1)}\\
\star & \star & \omega_{P,2n} & \cdots & \omega_{P,(3n-3)}\\
\vdots & \vdots & \vdots & \ddots  & \vdots \\
\star  & \star & \star & \cdots & \omega_{P,\left(\frac{n(n+1)}{2}\right)}
\end{bmatrix},\\
&\mathcal{B}_P=\begin{bmatrix}
\beta_{P,1} & \beta_{P,2} & \beta_{P,3} & \cdots & \beta_{P,n} \\
\star & \beta_{P,(n+1)} & \beta_{P,(n+2)} & \cdots & \beta_{P,(2n-1)}\\
\star & \star & \beta_{P,2n} & \cdots & \beta_{P,(3n-3)}\\
\vdots & \vdots & \vdots & \ddots  & \vdots \\
\star  & \star & \star & \cdots & \beta_{P,\left(\frac{n(n+1)}{2}\right)}
\end{bmatrix}.
\end{align*}



There are two constraint neurons (shown in green in Fig.~\ref{fig:NOMIP}) in the NOM constraint layer to represent the constraints  \eqref{eq:Constraint2} and \eqref{eq:Constraint3}. For each constraint, a Rectified Linear Unit (ReLU) activation function is applied to enforce constraint satisfaction. More precisely, the output of constraint neurons are:
\begin{subequations}
\begin{align}
G_1=\left\{
\begin{array}{rl}
    0, & \Delta V+\theta\left\Vert x-\bar{x}_r\right\Vert\leq0 \\
    c\cdot\left(\Delta V+\theta\left\Vert x-\bar{x}_r\right\Vert\right), & \text{Otherwise}
\end{array}
\right.,
\end{align}
and
\begin{align}
G_2=\left\{
\begin{array}{rl}
    0, & P\succ0 \\
    c\cdot\left\vert\lambda_{\min}(P)\right\vert, & \text{Otherwise}
\end{array}
\right.,
\end{align}
where $c\in\mathbb{R}_{>0}$ is a positive (typically large) constant and $\Delta V:=V(x^+,r,P)-V\left(x,r,P\right)$, with $x^+$ being as in \eqref{eq:Constraint1}.
\end{subequations}

\begin{figure}[!t]
    \centering
    \includegraphics[width=8.5cm]{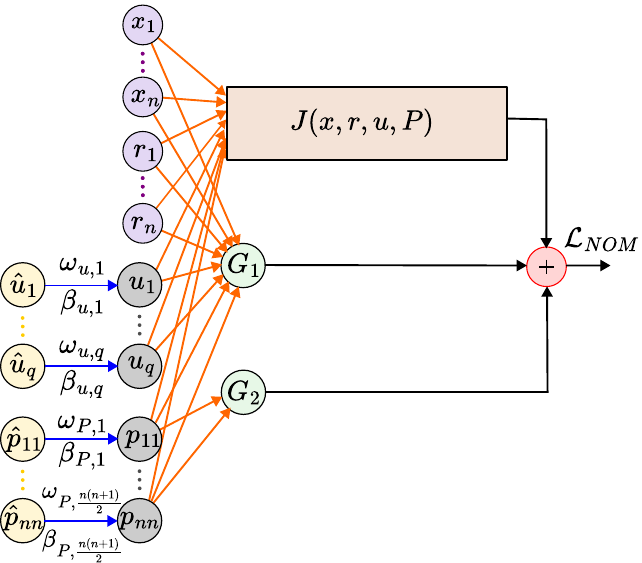}
    \caption{The general structure of the NOM proposed to obtain the solution of the optimization problem \eqref{eq:OptimizationProblemMain}.}
    \label{fig:NOMIP}
\end{figure}



The NOM output (the red neuron in Fig. \ref{fig:NOMIP}) is the sum of the objective function and the output of the constraint neurons, which is considered as the {NOM loss function}. In mathematical terms, the loss function is:
\begin{equation}\label{eq:NOMLoss}
\mathcal{L}_{NOM}(u,P)=J\left(x,r,u,P\right)+G_1+G_2.
\end{equation}

\noindent\textbf{NOM Training:} The main feature of the proposed NOM is that it uses the NNs’ built-in backpropagation algorithm to transform the problem of calculating the gradient with respect to inputs (i.e., $u$ and $P$) to the problem of calculating the gradient with respect to weights and biases.

For a given state vector $x\in\mathcal{X}$ and desired reference $r\in\mathcal{R}$, consider $S$ starting points (denoted by $\hat{u}^{(j)}$ and $\hat{P}^{(j)},~j\in\{1,\cdots,S\}$), and set the initial weights and biases for the $j$th starting point as follows:
\begin{subequations}\label{eq:RandomnessWeights}
\begin{align}
{\omega_{u,i}}^{(j)}=&1+\epsilon,~i\in\{1,\cdots,q\},\\
{\omega_{P,i}}^{(j)}=&1+\epsilon,~i\in\left\{1,\cdots,\frac{n(n+1)}{2}\right\},\\
{\beta_{u,i}}^{(j)}=&\epsilon,~i\in\{1,\cdots,q\},\\
{\beta_{P,i}}^{(j)}=&\epsilon,~i\in\left\{1,\cdots,\frac{n(n+1)}{2}\right\},
\end{align} 
\end{subequations}
where $\epsilon$ is a small random number; we use different random values for every weight and bias. For any given $x\in\mathcal{X}$, the most intuitive approach to select $S$ starting points is: i) divide the $(u,P)$-space into multiple cells; ii) compute the value of the loss function given in \eqref{eq:NOMLoss} for each cell; and iii) select $S$ cells that yield the smallest values.

For the $j$th starting point, the NOM uses the gradient descent algorithm to minimize the value of the loss function by adjusting the weights and biases of the connections between the NOM starting layer and the input layer. Once the value of the loss function converges, the optimal solution obtained from the $j$th starting point can be computed as follows:
\begin{subequations}
\begin{align}
{u^\ast}^{(j)}=&{\Omega_u^\ast}^{(j)}\odot\hat{u}^{(j)}+{\mathcal{B}_u^\ast}^{(j)},\\
{P^\ast}^{(j)}=&{\Omega_P^\ast}^{(j)}\odot\hat{P}^{(j)}+{\mathcal{B}_P^\ast}^{(j)},
\end{align}   
\end{subequations}
where ${\Omega_u^\ast}^{(j)}$, ${\mathcal{B}_u^\ast}^{(j)}$, ${\Omega_P^\ast}^{(j)}\odot\hat{P}^{(j)}$, and ${\mathcal{B}_P^\ast}^{(j)}$ are the optimal weights and biases. Once the optimal solutions from all starting points are computed, the solution to the optimization problem \eqref{eq:OptimizationProblemMain} can be computed as:
\begin{align}
u^{\ast},P^{\ast}=\arg\min\limits_{\substack{{u^\ast}^{(j)},{P^\ast}^{(j)},\\j\in\{1,\cdots,S\}}}\bigg\{&J\left(x,r,{u^\ast}^{(j)},{P^\ast}^{(j)}\right)\nonumber\\
&+G_1+G_2\bigg\}.
\end{align}

\begin{remark}
Although one starting point may be sufficient to train the NOM (note that the NOM is not trained in a supervised way), as suggested in \cite{Chen2024} and as observed in our numerical experiments, starting from different points and introducing randomness as in \eqref{eq:RandomnessWeights} can prevent the NOM from producing a locally optimal solution.    
\end{remark}

\section{Provably-Stable NN-Based Control}\label{sec:NNControl}
In this section, we discuss the data collection and training procedure to train a NN that emulates the behavior of the control scheme given in  \eqref{eq:OptimizationProblemMain} inside the region $\mathcal{X}$. Also, we analyze the properties of the closed-loop system where the trained NN is utilized in the loop to control system \eqref{eq:system}.

\noindent\textbf{Collecting Data:} To obtain the training dataset, first, we divide the operating region $\mathcal{X}$ and the set of steady-state admissible references $\mathcal{R}$ into $N_x$ and $N_r$ grid cells, respectively; this process generates $N_x\cdot N_r$ data points which are denoted by $(x^i,r^i),~i=1,\cdots,N_x\cdot N_r$, where $x^i\in\mathcal{X}$ and $r^i\in\mathcal{R}$. Then, for each data point $(x^i,r^i)$, we use the NOM discussed in Section \ref{sec:NOM} to obtain the optimal solution to the optimization problem \eqref{eq:OptimizationProblemMain}; we denote obtained control input by $u^i$ and the obtained Lyapunov matrix by $P^i$. Finally, the training dataset can be constructed as follows:
\begin{align}
\mathbb{D}_{training}=\left\{\left(x^i,r^i,u^i,P^i\right),~i=1\cdots,N_x\cdot N_r\right\},
\end{align}
where the tuple $\left(x^i,r^i,u^i,P^i\right)$ is the $i$th training data point.

\begin{remark}
By using a larger $N_x$ and $N_r$, one can make sure that enough training data is collected and the training dataset $\mathbb{D}_{training}$ covers all relevant aspects of the problem.
\end{remark}

\noindent\textbf{NN Training:} Once the training dataset $\mathbb{D}_{training}$ is collected, we use it to train a (deep) feedforward NN \cite{Fine1999}; our motivation to consider feedforward NNs is that they are flexible and scalable, and their training is relatively straightforward compared to other structures \cite{Goodfellow2016}. 
We use the growing method (see, e.g., \cite{Evci2022,Wu2020,Bengio2005}) to determine the number of hidden layers and neurons; that is, we start with a small network and increase the complexity until the desired accuracy is achieved. Also, we use the dropout method \cite{srivastava2014dropout,lim2021study,sanjar2020weight,wu2015towards} to avoid overfitting. To improve the training process, we use the cosine annealing strategy \cite{liu2022super,eshraghian2022navigating} to adjust the learning rate. To evaluate the NN's convergence, we use the mean squared error (MSE) loss function \cite{Goodfellow2016}.

\noindent\textbf{NN-based Control System:} Consider system \eqref{eq:system}, and suppose that a NN trained as discussed above is utilized in the control loop to imitate the behavior of the control scheme \eqref{eq:OptimizationProblemMain}. The following theorem studies the theoretical properties of the resulting closed-loop system. 
 
\begin{theorem}\label{theorem:NNControl}
Consider system \eqref{eq:system}, and suppose that a NN trained as mentioned above is utilized in the control loop. Then, given $r\in\mathcal{R}$ and a sufficiently large $\theta$, the set $\Phi=\{x|\left\Vert x-\bar{x}_r\right\Vert\leq\sigma\}$ is attractive and positively invariant, for some $\sigma\in\mathbb{R}_{>0}$. 
\end{theorem}

\begin{proof}
Let $\Delta V(t):=V\big(x(t+1),r,\tilde{P}(t+1)\big)-V\big(x(t),r,\tilde{P}(t)\big)$, where $\tilde{P}(t+1)=P^\ast(t+1)+\Delta P(t+1)$ and $\tilde{P}(t)=P^\ast(t)+\Delta P(t)$ are the output of the NN at time instants $t+1$ and $t$, respectively, with $\Delta P(t+1)$ and $\Delta P(t)$ being the NN's approximation error at corresponding time instants. According to the definition of the Lyapunov function $V$, we have:
\begin{align}\label{eq:DVNN1}
\Delta V(t)=&\left\Vert f\left(x(t)\right)+g\left(x(t)\right)\tilde{u}(t)-\bar{x}_r\right\Vert_{\tilde{P}(t+1)}\nonumber\\
&-\left\Vert x(t)-\bar{x}_r\right\Vert_{\tilde{P}(t)},
\end{align}
where $\tilde{u}(t)=u^\ast(t)+\Delta u(t)$ is the output of the NN at time instant $t$, with $\Delta u(t)$ being the NN's error in approximating the optimal control input $u^\ast(t)$. According to the triangle inequality, \eqref{eq:DVNN1} implies that:
\begin{align}\label{eq:DVNN2}
\Delta V(t)\leq&\left\Vert f\left(x(t)\right)+g\left(x(t)\right)u^\ast(t)-\bar{x}_r\right\Vert_{\tilde{P}(t+1)}\nonumber\\
&+\left\Vert g\left(x(t)\right)\Delta{u(t)}\right\Vert_{\tilde{P}(t+1)}-\left\Vert x(t)-\bar{x}_r\right\Vert_{\tilde{P}(t)}. 
\end{align}

From \eqref{eq:DVNN2}, it can be concluded that\footnote{Given $z\in\mathbb{R}^n$ and $Q_1,Q_2\in\mathbb{R}^{n\times n}$, let $\left\Vert z\right\Vert_{Q_1}=\sqrt{\left\vert z^\top Q_1z\right\vert}$, $\left\Vert z\right\Vert_{Q_2}=\sqrt{\left\vert z^\top Q_2z\right\vert}$, and $\left\Vert z\right\Vert_{Q_1+Q_2}=\sqrt{\left\vert z^\top (Q_1+Q_2)z\right\vert}$. Thus, it can be easily shown that $\left\Vert z\right\Vert_{Q_1}-\left\Vert z\right\Vert_{Q_2}\leq\left\Vert z\right\Vert_{Q_1+Q_2}\leq\left\Vert z\right\Vert_{Q_1}+\left\Vert z\right\Vert_{Q_2}$. Also, $\left\Vert z\right\Vert_{Q_1}\leq\sqrt{\left\Vert Q_1\right\Vert}\left\Vert z\right\Vert$ and $\left\Vert z\right\Vert_{Q_2}\leq\sqrt{\left\Vert Q_2\right\Vert}\left\Vert z\right\Vert$.}:
\begin{align}\label{eq:DVNN3}
\Delta V(t)\leq&\left\Vert f\left(x(t)\right)+g\left(x(t)\right)u^\ast(t)-\bar{x}_r\right\Vert_{P^\ast(t+1)}\nonumber\\
&+\sqrt{\bar{\Delta{P}}}\left\Vert f\left(x(t)\right)+g\left(x(t)\right)u^\ast(t)-\bar{x}_r\right\Vert\nonumber\\
&+\left\Vert g\left(x(t)\right)\Delta{u(t)}\right\Vert_{\tilde{P}(t+1)}\nonumber\\
&-\left\Vert x(t)-\bar{x}_r\right\Vert_{P^\ast(t)}+\sqrt{\bar{\Delta{P}}}\left\Vert x(t)-\bar{x}_r\right\Vert, 
\end{align}
where $\bar{\Delta P}=\sup_{t\geq0}\left\Vert\Delta P(t)\right\Vert$. Since $\big\Vert f\left(x(t)\right)+g\left(x(t)\right)u^\ast(t)-\bar{x}_r\big\Vert_{P^\ast(t+1)}-\left\Vert x(t)-\bar{x}_r\right\Vert_{P^\ast(t)}\leq3\sqrt{\bar{\lambda}_P}\delta-\theta\left\Vert x(t)-\bar{x}_r\right\Vert$ (see \cite[Theorem 2]{onestep2024}), it follows from \eqref{eq:DVNN3} that:
\begin{align}\label{eq:DVNN4}
\Delta V(t)\leq&{3\sqrt{\bar{\lambda}_P}\delta}-\theta\left\Vert x(t)-\bar{x}_r\right\Vert+\sqrt{\bar{\Delta{P}}}\left\Vert x(t)-\bar{x}_r\right\Vert\nonumber\\
&+\left\Vert g\left(x(t)\right)\Delta{u(t)}\right\Vert_{\tilde{P}(t+1)}\nonumber\\
&+\sqrt{\bar{\Delta{P}}}\left\Vert f\left(x(t)\right)+g\left(x(t)\right)u^\ast(t)-\bar{x}_r\right\Vert,
\end{align}
where $\bar{\lambda}_P$ and $\delta$ are as in \eqref{eq:TrackingPerformance1}. 

At this stage, following arguments similar to \cite{onestep2024}, it can be shown that:
\begin{align}
&\left\Vert g\left(x(t)\right)\Delta{u(t)}\right\Vert_{\tilde{P}(t+1)}\leq\left(\sqrt{\bar{\lambda}_P}+\sqrt{\bar{\Delta{P}}}\right)\left\Vert g(\bar{x}_r)\right\Vert\Delta\bar{u}\nonumber\\
&+\left(\sqrt{\bar{\lambda}_P}+\sqrt{\bar{\Delta{P}}}\right)\mu_g\Delta\bar{u}\left\Vert x(t)-\bar{x}_r\right\Vert,\label{eq:UpperBound1}
\end{align}
and
\begin{align}
&\sqrt{\bar{\Delta{P}}}\left\Vert f\left(x(t)\right)+g\left(x(t)\right)u^\ast(t)-\bar{x}_r\right\Vert\leq\sqrt{\bar{\Delta{P}}}\delta\nonumber\\
&+\sqrt{\bar{\Delta{P}}}\frac{\sqrt{\bar{\lambda}_P}}{\sqrt{\underline{\lambda}_P}}\left\Vert x(t)-\bar{x}_r\right\Vert,\label{eq:UpperBound2}
\end{align}
where $\underline{\lambda}_P:=\inf_{t\geq0}\lambda_{\text{min}}\left(P^\ast(t)\right)$ ($\underline{\lambda}_P\in\mathbb{R}_{>0}$), and $\bar{\Delta u}=\sup_{t\geq0}\left\Vert\Delta u(t)\right\Vert$.

Combining \eqref{eq:DVNN4}, \eqref{eq:UpperBound1}, and \eqref{eq:UpperBound2} yields:
\begin{align}               
\Delta{V(t)}\leq&{3\sqrt{\bar{\lambda}_P}\delta}+\sqrt{\bar{\Delta{P}}}\delta\nonumber\\
&+\left(\sqrt{\bar{\lambda}_P}+\sqrt{\bar{\Delta{P}}}\right)\left\Vert g(\bar{x}_r)\right\Vert\Delta\bar{u}\nonumber\\
&-\Bigg(\theta-\sqrt{\bar{\Delta{P}}}-\frac{\sqrt{\bar{\Delta{P}}}\sqrt{\bar{\lambda}_P}}{\sqrt{\underline{\lambda}_P}}\nonumber\\
&-\left(\sqrt{\bar{\lambda}_P}+\sqrt{\bar{\Delta{P}}}\right)\mu_g\Delta\bar{u}\Bigg)\left\Vert x(t)-\bar{x}_r\right\Vert.
\end{align}

Thus, by selecting $\theta>\sqrt{\bar{\Delta{P}}}+\frac{\sqrt{\bar{\Delta{P}}}\sqrt{\bar{\lambda}_P}}{\sqrt{\underline{\lambda}_P}}+\big(\sqrt{\bar{\lambda}_P}+\sqrt{\bar{\Delta{P}}}\big)\mu_g\Delta\bar{u}$, it is concluded that $\Delta{V(t)}<0$ whenever $\left\Vert x(t)-\bar{x}_r\right\Vert>\sigma$, where
\begin{align}\label{eq:sigma}
\sigma:=\frac{{3\sqrt{\bar{\lambda}_P}\delta}+\sqrt{\bar{\Delta{P}}}\delta+(\sqrt{\bar{\lambda_p}}+\sqrt{\bar{\Delta{P}}})\left\Vert g(\bar{x}_r)\right\Vert\Delta\bar{u}}{\theta-\sqrt{\bar{\Delta{P}}}-\frac{\sqrt{\bar{\Delta{P}}}\sqrt{\bar{\lambda_P}}}{\sqrt{\underline{\lambda}_P}}-(\sqrt{\bar{\lambda_p}}+\sqrt{\bar{\Delta{P}}})\mu_g\Delta\bar{u}}.
\end{align}

Therefore, since for any given $r\in\mathcal{R}$, $\left\Vert g(\bar{x}_r)\right\Vert$ is bounded, the set $\Phi=\{x|\left\Vert x-\bar{x}_r\right\Vert\leq\sigma\}$ is attractive and positively invariant, where $\sigma$ is as in \eqref{eq:sigma}; this completes the proof. 
\end{proof}

\begin{remark}
Since $\theta$ is a design parameter, $\sigma$ given in \eqref{eq:sigma} can be made arbitrarily small. Also, according to \eqref{eq:sigma}, in the presence of an ideal NN (i.e., $\bar{\Delta u}=\bar{\Delta P}=0$), the properties of the main controller (i.e., inequality \eqref{eq:TrackingPerformance1}) are recovered. 
\end{remark}

\begin{remark}
As discussed above, $\theta$ must exceed a threshold determined by the network's approximation errors. To achieve this, the process begins with training a NN to approximate the control input and the Lyapunov matrix. The maximum approximation errors are then computed, and $\theta$ is adjusted accordingly. The NN is subsequently re-trained with the updated $\theta$ value. This procedure is repeated until $\theta$ satisfies the specified inequality.
\end{remark}

\section{Simulation Study}\label{sec:SIM}

Consider the following nonlinear system:
\begin{subequations}\label{eq:sim_nonlinear}
\begin{align}
\dot{x}_1&=x_2,\\
\dot{x}_2&=x_1^3+(x_2^2+1)u,
\end{align}
\end{subequations}
which can be discretized using Euler's method as follows:
\begin{subequations}\label{eq:sim_nonlinear}
\begin{align}
x_1(t+1)&=x_1(t)+\Delta T x_2(t),\\
x_2(t+1)&=x_2(t)+\Delta T\big(x_1(t)^3+ (x_2(t)^2+1)\big)u(t),
\end{align}
\end{subequations}
where $\Delta T=0.1$ seconds is the sampling period. We assume that $\mathcal{X}=[-5,5]\times[-5,5]$ and $\mathcal{R}=0$, and we set $N_x=10,201$ (i.e., $|\mathbb{D}_{training}|=10,201$). Also, we set $\theta=0.1$, $Q_x=\text{diag}\{10,0.1\}$, and $Q_u=1$. For each data point, we use 15 NOM starting points (i.e., $S=15$). The NOM for all starting points is trained in 2000 epochs with a learning rate of 0.01. For comparison purposes, we also use YALMIP toolbox \cite{Lofberg2004} with \texttt{fmincon} solver to solve \eqref{eq:OptimizationProblemMain} for each data point. This experiment is conducted on $\text{Intel}^{\text{\textregistered}}$ $\text{Core}^{\text{\texttrademark}}$ i5-7300HQ processor with 24GB RAM, using a GeForce GTX 1050 Ti graphic to accelerate the NOM computations.

TABLE \ref{tab:pfm} compares NOM and YALMIP in solving the optimization problem \eqref{eq:OptimizationProblemMain} for system \eqref{eq:sim_nonlinear}, where Control Effort (CE) is defined as $\text{CE}=\sum\left\vert u(t)\right\vert$ and the CE with YALMIP is used as the basis for normalization. As seen in this table, YALMIP leads to infeasibility in 46.31\% of instances, despite the fact that the problem is feasible. Also, YALMIP terminates prematurely and provides an infeasible solution in 18.29\% of instances. TABLE \ref{tab:pfm} reveals that YALMIP increases the CE by {431.91\%}, as it cannot solve the optimization problem \eqref{eq:OptimizationProblemMain} effectively.

We also compare the performance of the NN-based control scheme where the NN is trained on the dataset generated by NOM and YALMIP. The NN structure for both datasets has 6 hidden layers with 8, 32, 64, 64, 32, and 16 neurons. For training the NNs, we consider the MSE loss function and use the Adam optimizer with a learning rate of 0.001; both NNs converge in 10,000 epochs. For comparison purposes, we also implement the iterative LQR detailed in \cite{Prasad2014}.

Starting from the initial condition $x=[1~0]^\top$,  Fig. \ref{fig:timeprofile_sim} presents the time-profile of states and control inputs with all control schemes. As seen in this figure, the NN-based control scheme developed based on the dataset generated by NOM provides a comparable performance to the iterative LQR. Since YALMIP does not provide a good training dataset, the NN-based control scheme developed based on the YALMIP-generated dataset yields a poor performance.



\begin{table}[!t]
    \normalsize
    \centering
    \caption{Comparing the Performance of NOM and YALMIP in solving \eqref{eq:OptimizationProblemMain} for System \eqref{eq:sim_nonlinear}.}\label{tab:pfm}
    \begin{tabular}{c|c|c|c}
        \hline
        Method & Infeasibility  & Violation & CE\\
        & ($\%$) & ($\%$) & (Norm.) \\
        \hline
        NOM & 0 & 0  & 0.188\\
        \hline
        YALMIP & 46.31 & 18.29 & 1 \\
        \hline
    \end{tabular}
\end{table}

\begin{figure}[!t]
    \centering
    \includegraphics[width=8.5cm]{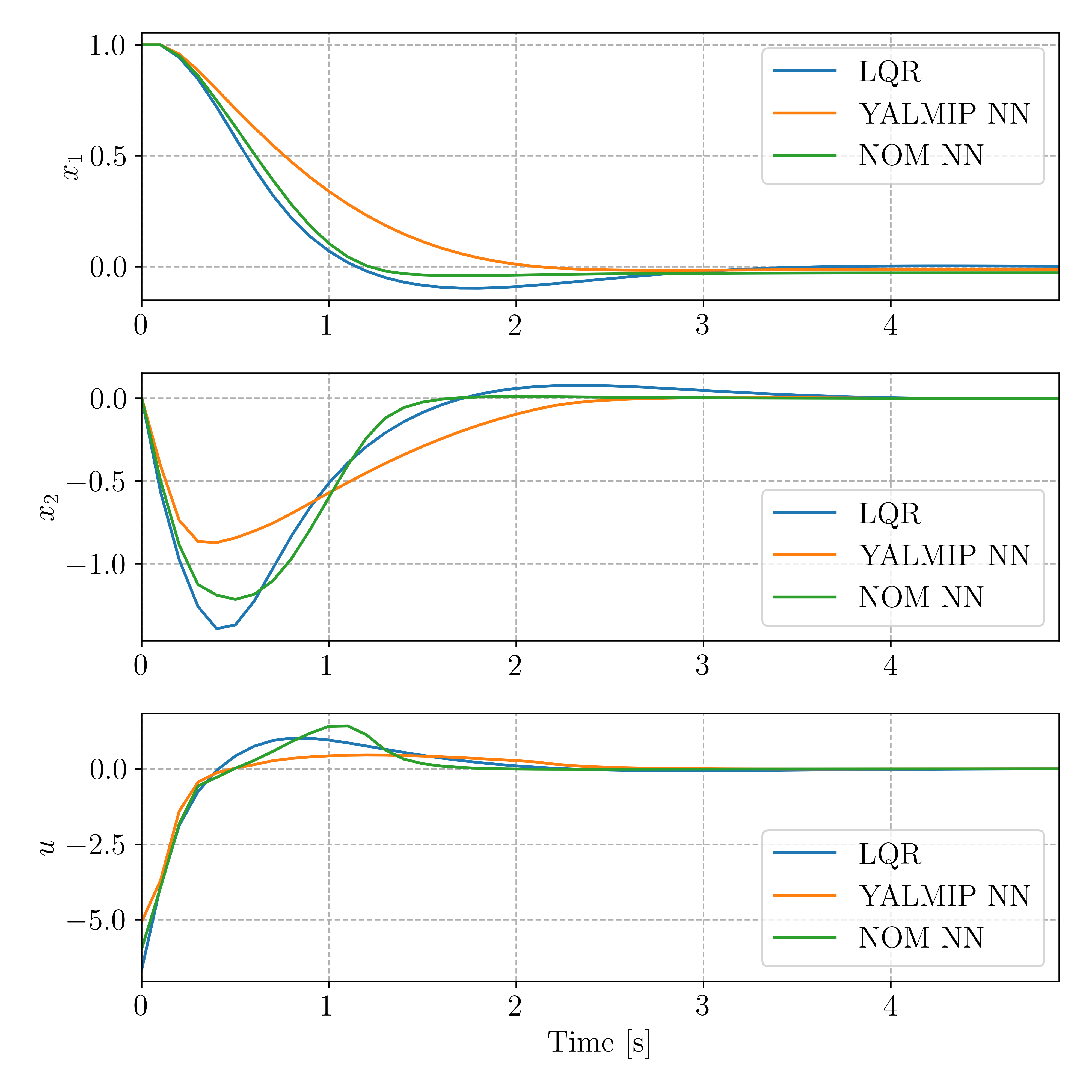}
    \caption{Time-profile of states and control input with NN-based controller developed based on the NOM-generated dataset (indicated by NOM NN), NN-based controller developed based on the YALMIP-generated dataset (indicated by YALMIP NN), and iterative LQR \cite{Prasad2014} (indicated by LQR).}
    \label{fig:timeprofile_sim}
\end{figure}

\section{Experimental Results}\label{sec:EXPRIMENT}
This section uses the proposed methodology to control the position of a Parrot Bebop 2 drone, whose dynamics are reported in \cite{AmiriMECC}. The training set $\mathbb{D}_{\text{training}}$ is collected for X, Y, and Z directions by following the procedure detailed in \cite{onestep2024}, where  NOM is used to solve the optimization problem \eqref{eq:OptimizationProblemMain}. Note that the dynamics along X, Y, and Z directions are decoupled, allowing us to train NN for each direction separately. For all directions, we set $Q_x=\text{diag}\{20,0.1\}$, $Q_u=0.1$, and $\theta=0.01$. We use a feedforward NN with 6 hidden layers, and 8, 32, 64, 64, 32, and 16 neurons in the hidden layers.

Our experimental setup is shown in Fig. \ref{fig:network}. We use \texttt{OptiTrack} system with ten \texttt{Prime$^\text{x}$ 13} cameras operating at a frequency of 120 Hz; these cameras provide a 3D accuracy of $\pm0.02$ millimeters which is acceptable for identification purposes. The computing unit is a $13^{\text{th}}$ Gen $\text{Intel}^{\text{\textregistered}}$ $\text{Core}^{\text{\texttrademark}}$ i9-13900K processor with 64GB RAM and one GeForce GTX 3090 graphic card, on which the software \texttt{Motive} is installed to analyze and interpret the camera data. We use the ``Parrot Drone Support from MATLAB" package \cite{MATLAB}, and send the control commands to the Parrot Bebop 2 via WiFi and by using the command \texttt{move($\cdot$)}. It should be mentioned that the communication between \texttt{Motive} and MATLAB is established through User Datagram Protocol (UDP) communication using the \texttt{NatNet} service.

For comparison purposes, in addition to running networks trained on the NOM-generated  dataset, we implement networks trained on the YALMIP-generated dataset; we also implement the iterative LQR reported in \cite{Prasad2014}. Fig. \ref{fig:timeprofile_drone} presents the results with all three schemes. In X and Y directions, both NN-based control schemes outperform the iterative LQR \cite{Prasad2014}. In Z direction, the iterative 
LQR provides a better solution in comparison with the NN-based control schemes; this observation is reasonable, as the linearization error $\delta$ is small for Z direction. In Z direction, the NN-based control scheme developed based on the YALMIP-generated dataset yields the poorest performance. 

\begin{figure}[!t]
    \centering
    \includegraphics[width=8.5cm]{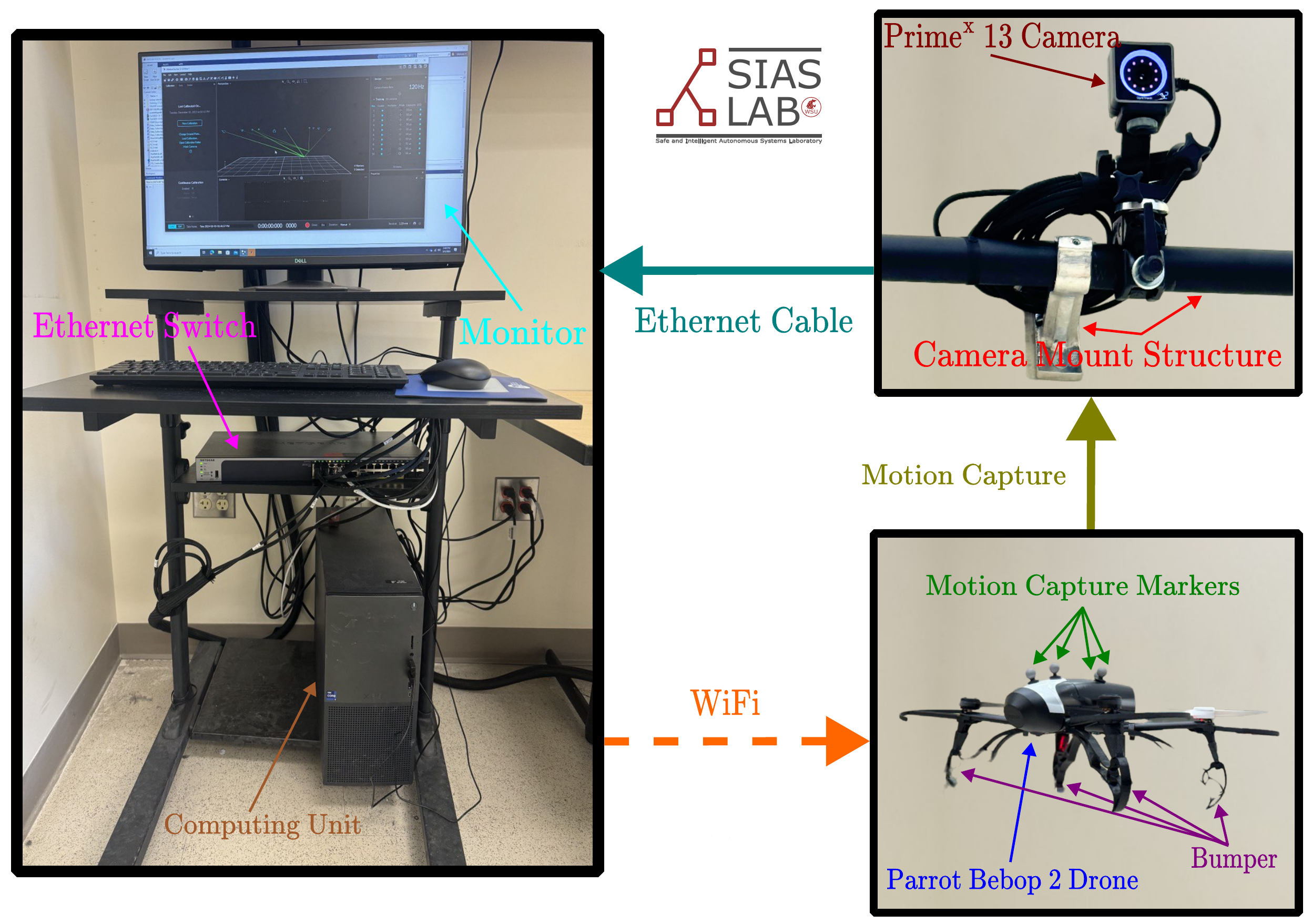}
    \caption{Overview of the experimental setup utilized to experimentally validate the proposed NN-based control scheme.}
    \label{fig:network}
\end{figure}

\begin{figure}[!t]
    \centering
    \includegraphics[width=8.5cm]{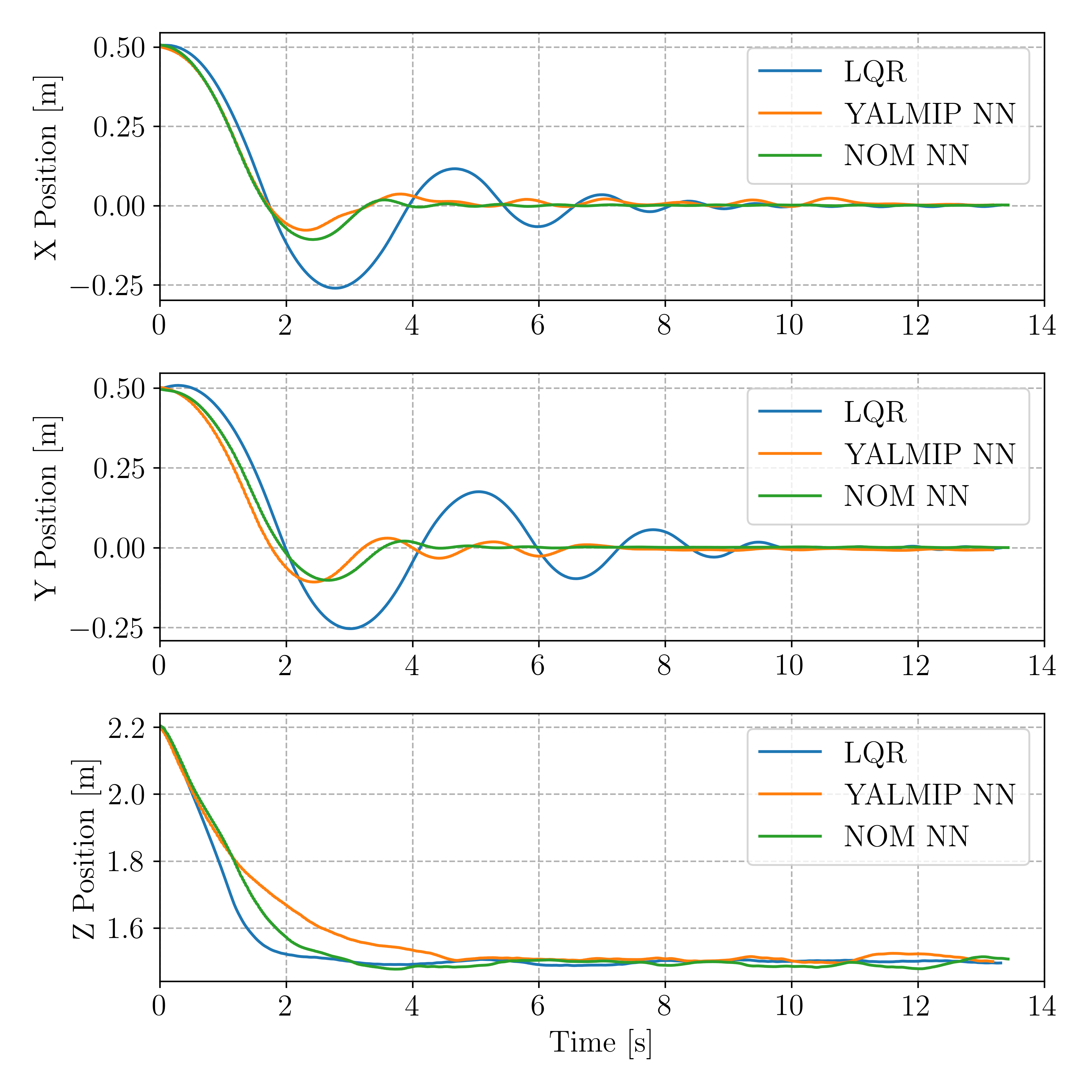}
    \caption{Time-profile of Parrot Bebop 2 with NN-based controller developed based on the NOM-generated dataset (indicated by NOM NN), NN-based controller developed based on the YALMIP-generated dataset (indicated by YALMIP NN), and iterative LQR \cite{Prasad2014} (indicated by LQR).}
    \label{fig:timeprofile_drone}
\end{figure}

\section{Conclusion}\label{sec:Con}
Prior work has introduced a provably-stable NN-based control scheme for the optimal control of nonlinear systems, where the NN in the control loop emulates a one-step-ahead predictive control policy. Despite its potential, applying this approach to real-world systems can be challenging due to optimization difficulties. Specifically, generating a reliable training dataset for the NN involves solving a non-convex, nonlinear optimization problem that can easily exceed the capabilities of existing solvers. This paper addressed this issue by developing a NOM to solve the optimization problem and generate a reliable and comprehensive training dataset. This paper discussed the NOM structure and its training details, and examined the theoretical properties of the closed-loop system where the NN, trained on the NOM-generated dataset, is used in the control loop. Simulation and experimental studies were presented to demonstrate the effectiveness of the proposed methodology.

\section*{Data Availability}
All the cade and data used in simulation and experimental studies are available at \url{https://github.com/anran-github/A-Guaranteed-Stable-Neural-Network-Approach-for-Optimal-Control-of-Nonlinear-Systems.git}. 

\balance
\bibliographystyle{IEEEtran}
\bibliography{references}{}

\end{document}